\documentclass[english]{ourlematema}
\usepackage[utf8]{inputenc}
\usepackage{xcolor}
\usepackage{comment}
\usepackage{hyperref}
\usepackage{multirow}

\usepackage[noend]{algorithmic} 

\newcommand{\AlgComment}[2]{\hspace{#1}\textit{\footnotesize \textcolor{black!70}{// #2}}}

\newcommand{\bZ}{\mathbb{Z}}

\newcommand{\bK}{\mathbb{K}}

\newcommand{\bR}{\mathbb{R}}
\newcommand{\bP}{\mathbb{P}}

\newcommand{\tgr}{\operatorname{TGr}(3,6)}  
\newcommand{\gr}{\operatorname{Gr}(3,6)} 
\newcommand{\grO}{\operatorname{Gr_0}(3,6)} 
\newcommand{\crxa}{R_A} 
\newcommand{\tb}{\operatorname{TB}}  
\newcommand{\tpr}{\overline{\pi}}  

\newcommand{\pl}{I_{pl}}
\newcommand{\rpl}{R_{pl}}

 \newtheorem{alg}[thm]{Algorithm}

\DeclareMathOperator{\init}{in}
\DeclareMathOperator{\trop}{trop}

\DeclareMathOperator{\Relint}{Relint}
\DeclareMathOperator{\Sym}{Sym}

\newcommand\restr[2]{{\left.\kern-\nulldelimiterspace #1 \right|_{#2}}}
\newcommand\supind[1]{{\smash[t]{(#1)}}}

\title{Towards classifying toric degenerations\\ of cubic surfaces}
\titlemark{Toric degenerations of cubic surfaces}

\titlenote{MDB was partially supported by the Polish National Science Center project 2017/26/D/ST1/00755.
  PG was supported by the International Max Planck Research School “Mathematics in the Sciences”.
  MW was supported by the "Ministerium für Wissenschaft, Forschung und Kunst Baden-Württemberg" and the ESF.}


\author{Maria Donten-Bury}
\address{%
Institute of Mathematics, University of Warsaw\\
Banacha 2, 02-097 Warszawa, Poland\\
\email{m.donten@mimuw.edu.pl}
}
\author{Paul G\"orlach}
\address{%
Max Planck Institute for Mathematics in the Sciences\\
Inselstraße 22, 04103 Leipzig\\
\email{goerlach@mis.mpg.de}
}
\author{Milena Wrobel}
\address{%
Carl von Ossietzky Universität Oldenburg\\
Institut für Mathematik, 26111 Oldenburg
\\
\email{milena.wrobel@uni-oldenburg.de}
}

\date{2019/09/14}

\begin{document}

\maketitle

\begin{abstract}
We investigate the class of degenerations of smooth cubic surfaces which are obtained
from degenerating their Cox rings to toric algebras.
More precisely, we work in the spirit
of Sturmfels and Xu who use the theory of Khovanskii bases to determine
toric degenerations of Del Pezzo surfaces of degree 4 and who
leave the question of classifying these degenerations in the degree 3 case as an open
problem. In order to carry out this classification we describe an approach
which is closely related to tropical geometry and present partial results in this direction.
\end{abstract}

\section{Introduction}
We work over the field of rational functions $K:=\mathbb{K}(t)$, where $\mathbb{K}$ is any field of characteristic zero. Our main objects are smooth cubic surfaces in $\mathbb{P}^3$ that arise as the blow up of $\bP^2$ in six points in general position. We store the coordinates of these points in a matrix $A$ and denote the corresponding cubic surface with $X_A$. 
In this article we investigate \emph{toric degenerations} of these varieties, i.e.\ flat families with general fiber $X_A$ such that the special fiber is a toric~variety. Degenerations are a powerful tool in algebraic geometry as they allow 
to deduce geometric properties of the general fiber, as for example the Hilbert function, from the special one.

Our approach to find degenerations 
is based on the article 
\cite{StXu} of Sturmfels and Xu 
who consider degenerations of the 
spectrum of the \emph{Cox ring}
$$
\mathcal{R}(X_A) = \bigoplus_{[D] \in \mathrm{Cl}(X_A)}\Gamma(X_A, \mathcal{O}_{X_A}(D)),
$$
of varieties $X_A$ arising as the blow up of $\bP^2$ 
in at most $8$ points in general position, where $A$ is as before the matrix storing the coordinates of these points.
Note that one can regain $X_A$ as a GIT-quotient of its \emph{total coordinate space}, i.e.\ the spectrum of its Cox ring.
The crucial property of the Cox rings of these varieties is that due to a result of Nagata \cite{Nagata} 
they are isomorphic to an invariant ring of a certain group action, the so called \emph{Cox-Nagata ring} $R_A$, see Section \ref{sec2} for the construction.

In particular we obtain a realization of $\mathcal{R}(X_A)$ as a subalgebra of a polynomial ring over the field $K$. This allows us to find toric degenerations of $\mathrm{Spec} \, \mathcal{R}(X_A)$ via the theory of Khovanskii bases. (Note that not all toric degenerations arise in this way.)
The notion of Khovanskii bases was developed in \cite{RobbSwee} by Robbiano and Sweedler, where they are referred to as \emph{sagbi bases}, and was generalized and made applicable in a much broader setting in the article \cite{KaMaKhova} of Kaveh and Manon.

Let us fix the setting:
For every $p \in \mathbb{K}(t)$ let $\omega_p$ be the unique integer such that
$t^{- \omega_p}p(t)$ has neither a pole nor a zero at $t=0$. Consider the valuation
$$
\nu \colon \mathbb{K}(t)^* \rightarrow \bZ,
\quad
p \mapsto \omega_p.$$
Then for each $f \in K[x_1\ldots, x_r]$ we define its
\emph{initial form} to be 
$$\mathrm{in}(f):= (t^{- \omega}f)|_{t=0} \in \bK[x_1\ldots, x_r],$$
where $\omega$ is the minimum of $\nu$ on the coefficients of all monomials of $f$. 
For a $K$-subalgebra $U \subseteq K[x_1, \ldots, x_r]$ 
we define its \emph{initial algebra} $\mathrm{in}(U)$ to be the $\mathbb{K}$-subalgebra of $\bK[x_1, \ldots, x_r]$ generated by $\mathrm{in}(f)$, where $f$ runs over all elements of $U$. 
\begin{dfn}
In the above setting we call a subset $\mathcal{F}:=\left\{f_1, \ldots, f_m\right\} \subseteq U$ \emph{moneric of weight  $(\omega_1, \ldots, \omega_m)$} 
if $\mathrm{in}(f_i) = (t^{- \omega_i}f_i)|_{t=0}$ is a monomial 
for all $i$. If furthermore
the initial algebra 
$\mathrm{in}(U)$ is generated by the set 
$\left\{\mathrm{in}(f); \ f \in \mathcal{F}\right\}$ we call $\mathcal{F}$ a \emph{Khovanskii basis} of $U$.
\end{dfn}

\noindent Having a finite Khovanskii basis on hand 
one obtains a degeneration of
$\mathrm{Spec}(U)$ to the toric variety $\mathrm{Spec}(\mathrm{in}(U))$, 
see e.g.\ \cite[Thm.3.3]{BeCoDBFM}.
Note that the set of degenerations arising this way 
strictly contains the set of \emph{Gr\"obner degenerations} to a prime ideal, see \cite[Thm. 1]{KaMaKhova} and \cite[Sect. 1]{CoHeVa}.
See \cite[Ex. 3.3]{StXu} 
for an example of a degeneration arising 
via the above construction but which cannot be obtained as a Gr\"obner degeneration.

As by construction, the Cox ring and therefore the Cox-Nagata ring contains information about all possible embeddings of $X_A$ 
into projective space and $X_A$ can be regained as a GIT-quotient out of its total coordinate space 
one can deduce degenerations of $X_A$ from that of its Cox-Nagata ring.
More precisely, in \cite{BeCoDBFM}
Bernal et al. show that 
starting with a suitably high multiple of a very ample line bundle of $X_A$ the toric degeneration of the Cox-Nagata ring provides a toric degeneration of $X_A$ such that the special fiber and $X_A$ are 
embedded into the same projective space $\bP^N$.

The original motivation of Sturmfels and Xu
to study this type of degenerations was the following:
Any toric degeneration of the total coordinate space of $X_A$ 
yields an
Erhart-type formula for the Hilbert function of the Cox ring $\mathcal{R}(X_A)$. 
This means that it can be read off in purely combinatorial terms by counting lattice points in slices
of the rational polyhedral cone describing the special fiber.
Moreover, fixing an embedding of $X_A$ into projective space the formula for the Hilbert function of the Cox ring induces a realization of the Hilbert function of $X$ with respect to the given embedding as the Ehrhart function of an explicit rational convex polytope, see \cite[Chap. 4]{BeCoDBFM}.

As the Hilbert functions of the Cox rings 
$\mathcal{R}(X_A)$ for generically chosen points 
coincide, the following natural question arises: 

\vspace{3mm}
\noindent
\textbf{Question:}
Which choice of points on
$\bP^2$
is the optimal one
for computing an Ehrhart-type formula for the Hilbert function
of the Cox-Nagata ring $R_A$, i.e.\
which of the corresponding cones 
of the special fiber
is defined by the smallest number of linear 
inequalities? 
\vspace{3mm}

Sturmfels and Xu answer this question
in the case of total coordinate spaces of blow ups of 
$\bP^2$ in five points and leave
the question in 
the cubic surface case as an open problem,
which is the starting point of this article.
It is worth noting that toric degenerations and related Ehrhart-type formulas may be used to investigate certain properties of the Cox-Nagata ring. For example, in \cite{StXu}, they are the key ingredient in the proof of the Batyrev--Popov conjecture that Cox rings of del Pezzo surfaces are presented by ideals of quadrics.

We now describe the organization of this article
and summarize the main results:
In Section \ref{sec2} we give the necessary background on Cox-Nagata rings and explain 
the previous results of \cite{BeCoDBFM} and \cite{StXu} on this topic.
Moreover we describe an ad hoc approach to answer the above question. As it turns out that this approach is computationally too expensive to 
finish in a reasonable time
we describe in Section \ref{SecIdea} an alternative approach
for finding a computable fan parametrizing moneric and Khovanskii subspaces.
The idea is to split the problem into two steps.
The first step gives rise to a partial classification of the possible degenerations and is treated in Section \ref{secResults}. 
Here we obtain $78$ moneric classes
$38$ of which are Khovanskii; see Theorems \ref{thm:SubdivTGr} and \ref{thmKhovanskii}
for the precise result.
The second step is more elaborate
and
ends up in a serious case by case study.
We present the main ideas and state a complete result
for one of seven possible cases in Section \ref{sec:EEEE} Theorem \ref{thm:mainthm5}. 
In order to obtain these results we make extensive use
of the computer algebra system \texttt{Singular}~\cite{Singular} and, for the results of Section~\ref{secKhovanskii}, \texttt{Macaulay2}~\cite{M2}. 
The code, extensive classification output and supplementary material is made available at \url{https://software.mis.mpg.de}.

\section{Earlier results}\label{sec2}
In this section we give the general construction of the Cox-Nagata ring which represents 
the Cox ring of a projective space $\bP^r$ blown up in $n$ points in general position. We recall how this ring was used in \cite{StXu} to compute toric degenerations of smooth del Pezzo surfaces and to give an Erhart-type formula for the Hilbert polynomial.
Moreover, we introduce the tropical Grassmannian and recall the results of \cite{BeCoDBFM}, where it is shown that this natural candidate is insufficient for our task.

Let us recall how to compute the Cox ring of projective $r$-space blown up in $n$ points in general position:
Let $A:= (a_1, \ldots, a_n)$ be an $(r + 1) \times n$ matrix of full rank over $K$ with pairwise linearly independent columns and denote by $X_A$ the blow up of $\bP^{r}$ in the corresponding points. 
Denote with $D_1, \ldots, D_n \subseteq X$ the exceptional divisors and let $D_0 \subseteq X$ be the proper transform of a hyperplane. 
Then the subgroup $\mathcal{D} \subseteq \mathrm{Div}(X_A)$ generated by the $D_i$ projects isomorphically onto the free abelian group $\mathrm{Cl}(X_A)$ and the Cox ring of $X_A$ is given as 
$$
\mathcal{R}(X_A) = \bigoplus_{D \in \mathcal{D}} \Gamma(X_A, \mathcal{O}_{X_A}(D)).
$$
As shown by Nagata in \cite{Nagata} this Cox ring has an especially nice interpretation in terms of the invariant ring of the following group action: Let
$G \subseteq K^{n}$ be the nullspace of
$A$. Then $G$ is an (additive) unipotent group and
setting 
$\lambda \cdot z := (\lambda_1z_1, . . . , \lambda_nz_n),$
for a point $z \in K^n$
we obtain an action
$$
G \times (K^n \times K^n) \rightarrow (K^n \times K^n)
\quad (\lambda,(z, w)) \mapsto (z, w + \lambda \cdot z).$$
We set $R:= K[x_1, \ldots, x_n, y_1, \ldots, y_n]$
for the polynomial ring whose variables $x_i$ resp.\ $y_j$ correspond to the coordinates $z_i$ resp.\ $w_j$. Then the ring of invariants $R_A:=R^G$ 
is called the \emph{Cox-Nagata ring}
and the following statement holds:

\begin{thm}\label{thm:CoxNagata}
Assume $r \geq 2$ holds. Then the Cox ring $\mathcal{R}(X_A)$ is isomorphic to the Cox-Nagata ring $R_A$.
\end{thm}

In \cite{Mukai} Mukai explicitely describes the isomorphism of Theorem \ref{thm:CoxNagata}. Moreover, in \cite{CaTe}  Castravet and Tevelev give an 
explicit set of generators of the Cox ring of a projective space $\bP^r$ blown up in at least $r+3$ points lying on a rational normal curve of degree $r$. 
This generalizes a result of Batyrev and Popov \cite{BaPo} who show that the Cox ring of a smooth del Pezzo surface of degree at least $2$ arising as a blown up $\bP^2$ is generated by global sections corresponding to the exceptional curves. 
Using this, we obtain a distinguished set of generators of $R_A$, which is up to scalar multiples determined by
the linear subspace $G \subseteq K^n$. 
In the following we will always assume $R_A$ to be represented by this set of generators $\mathcal{F}$ and 
call the defining linear subspace $G \subseteq K^n$ \emph{moneric (Khovanskii) subspace} if the corresponding set of generators is  a moneric (a Khovanskii basis). 
This enables us to reformulate our aim as follows, see Problem 5.4 in \cite{StXu}:

\vspace{3mm}
\noindent
\textbf{Aim:}
Determine all equivalence classes of $3$-dimensional Khovanskii subspaces of~$K^6$,
where two subspaces $G$ and $G'$ are called equivalent if the corresponding initial algebras $\mathrm{in}(R^G)$ and $\mathrm{in}(R^{G'})$ coincide.
\vspace{3mm}

Let us explicitly look at the generators of the Cox-Nagata ring of
the smooth cubic del Pezzo surfaces. Let $G$ and $A$ be as before with $n=6$ and $r = 3$ and assume that the points in $\bP^2$ corresponding to the columns of $A$ are in general position, i.e.\ no three of them lie on a line and no six of them on a conic. 
Denote by $p_{ijk}$ the Plücker coordinates of $G$.
Then the set of generators $\mathcal{F}$ of $R_A \subseteq K[x_1, \ldots, x_6, y_1, \ldots, y_6]$ consists of the following elements:
\begin{itemize}
    \item 6 generators corresponding to the exceptional divisors: $E_i = x_i$.
    \vspace{-3mm}
    \item 15 generators 
    corresponding to the lines through pairs of points, e.g.\ we have
    \begin{equation} \label{eq:genF}
    F_{12} \ = \ 
    p_{123}y_3x_4x_5x_6 + 
    p_{124}x_3y_4x_5x_6 +
    p_{125}x_3x_4y_5x_6 +
    p_{126}x_3x_4x_5y_6.
    \end{equation}
    \vspace{-6mm}
    \item 6 generators $G_1, \ldots, G_6$ corresponding to the quadrics through any five of the points, e.g.\ we have
    \begingroup
    \begin{align}
        G_1 \ = \ 
        &p_{234}p_{235}p_{236}p_{456}\cdot y_2y_3x_4x_5x_6x_1^2 
        + 
        p_{234}p_{246}p_{245}p_{356}\cdot y_2y_4x_3x_5x_6x_1^2 \nonumber
        \\
        + 
        &p_{235}p_{245}p_{256}p_{346}\cdot y_2y_5x_3x_4x_6x_1^2 
        +
        p_{236}p_{246}p_{256}p_{345}\cdot y_2y_6x_3x_4x_5x_1^2 \nonumber
        \\
        + 
        &p_{234}p_{345}p_{346}p_{256}\cdot y_3y_4x_2x_5x_6x_1^2 
        +
        p_{235}p_{345}p_{356}p_{246}\cdot y_3y_5x_2x_4x_6x_1^2 \nonumber
        \\
        +
        &p_{236}p_{346}p_{356}p_{245}\cdot y_3y_6x_2x_4x_5x_1^2 
        +
        p_{245}p_{345}p_{456}p_{236}\cdot y_4y_5x_2x_3x_6x_1^2 \nonumber
        \\
        +
        &p_{246}p_{346}p_{456}p_{235}\cdot y_4y_6x_2x_3x_5x_1^2 
        +
        p_{256}p_{356}p_{456}p_{234}\cdot y_5y_6x_2x_3x_4x_1^2 \nonumber
        \\
        +
        &(p_{235}p_{346}p_{124}p_{256} - p_{234}p_{356}p_{125}p_{246})\cdot y_2y_1x_3x_4x_5x_6x_1 \nonumber
        \\
        +
        &(p_{235}p_{246}p_{134}p_{356} - p_{234}p_{256}p_{135}p_{346})\cdot y_3y_1x_2x_4x_5x_6x_1 \nonumber
        \\
        +
        &(p_{245}p_{236}p_{134}p_{456} + p_{234}p_{256}p_{145}p_{346})\cdot y_4y_1x_2x_3x_5x_6x_1 \nonumber
        \\
        +
        &(p_{235}p_{246}p_{145}p_{356} - p_{245}p_{236}p_{135}p_{456})\cdot y_5y_1x_2x_3x_4x_6x_1 \nonumber
        \\
        +
        &(p_{236}p_{245}p_{146}p_{356} - p_{246}p_{235}p_{136}p_{456})\cdot y_6y_1x_2x_3x_4x_5x_1 \nonumber
        \\
        +
        &(p_{235}p_{246}p_{134}p_{156} - p_{234}p_{256}p_{135}p_{146})\cdot y_1^2x_2x_3x_4x_5x_6. \label{eq:genG}
    \end{align}
    \endgroup
    \vspace{-3mm}
    \noindent
\end{itemize}
Note that all the other generators can be obtained out of the ones given above by permuting the indices. 

To compute the initial forms of this set of generators we use the following observation regarding the isomorphism given in \cite[Sect. 2.3]{BeCoDBFM}: Fixing a linear subspace 
$G \in \mathrm{Gr}(n-r,n)$ the coefficients of the monomials in the generators of $R_A$ are
always given as polynomials in the Plücker coordinates of $G$. 
Moreover for generic $G$ none of these coefficients
equals zero.
Now set $N:= \binom{n}{n-r}$, let $p_1, \ldots, p_N$ denote the Plücker coordinates ordered lexicographically and let
$I_{pl}$ be the ideal generated by the Plücker relations. 
Then we may regard the generators of $R_A$ as 
elements of the ring $S:=\bK[x_1, \ldots, x_n,y_1, \ldots, y_n, p_1, \ldots, p_N] / I_{pl}$,
evaluated at the Plücker coordinates of $A$.
If it is clear from the context we will use  in the following the term "generator" 
as well for the elements in $S$ 
as for their evaluation at a point.

Recall that the value of a product is the sum of the values of its factors. In particular, given the distinguished set of generators of $R_A$, the values of the Plücker coordinates together with the values of all coefficients of the
generators of $R_A$ that are not monomials in Plücker coordinates uniquely determine the initial forms of the generating set.

To determine the initial forms of the generators $\mathcal{F}$ of 
$R_A$ we need to know all values of 
the occurring coefficients. Recall that the value of a product is the sum of the values of its factors. In particular, values of coefficients which are monomials in
the Pl\"ucker coordinates $p_i$ are already determined by the values of the $p_i$ themselves. 
In contrast, the values of the coefficients that are not monomials in Plücker coordinates are not uniquely determined by the values of the $p_i$. 
Hence we need to treat them separately and will denote them with $c_i$ for short.
In the case that $R_A$ is the Cox-Nagata ring of a cubic surface, the $c_i$ are precisely the binomial coefficients occurring in the generators $G_i$.

We deduce the following ad hoc approach for finding 
moneric or Khovanskii subspaces for Cox-Nagata rings 
$R_A$. 
Consider the map
\begin{equation}\label{equ1}
  G \mapsto (\nu(p_1), \ldots, \nu(p_N), \nu(c_1), \ldots, \nu(c_s)),  
\end{equation}
and proceed as follows:
\begin{itemize}
    \item[(i)] Determine the image of the above map.
    \vspace{-3mm}
    \item[(ii)] Subdivide the image into sets of points that give rise to the same initial forms of the generators.
    \vspace{-3mm}
    \item[(iii)] Identify those sets of points that arise from moneric or Khovanskii subspaces $G \subseteq K^n$.
\end{itemize}

We will see that the image of the above map 
consists of integer points of a certain tropical variety, a reembedding of the so-called tropical Grassmannian.
Let us recall the necessary notions on tropical varieties:
Let $k$ be a field with valuation $\mathrm{val}\colon k^* \rightarrow \mathrm{\bR}$
and let $X$ be a closed subvariety of the algebraic torus $\mathbb{T}^N(k)$. Then the \emph{Bieri-Groves set} $\mathcal{A}(X)$ of $X$ is 
$$
\mathcal{A}(X) := \left\{(\mathrm{val}(x_1),\ldots,\mathrm{val}(x_N)) \in \bR^N; \  (x_1,\ldots,x_N) \in X\right\}.$$
Let $\mathbb{L}$ be any algebraically closed field extension of $k$ with a nontrivial valuation extending the valuation on $k$. Then 
due to the Fundamental Theorem of Tropical Geometry
and \cite[Thm. 3.2.3]{MacLStur}
the closure of $\mathcal{A}(X_\mathbb{L})$ in $\bR^N$
equals the \emph{tropicalization} or \emph{tropical variety} $\mathrm{trop}(X)$,
where $X_\mathbb{L}$ denotes the extension of $X$ to a
closed subvariety of $\mathbb{T}^N(\mathbb{L})$.

We now specialize to the case of the tropical Grassmannian which will play an important role in our further considerations. 
Let $k = K$, embed $\mathrm{Gr}(d,n)$ 
with respect to the Plücker embedding into
$\bP^{N-1}$
and let $\mathrm{Gr}_0(d,n)$
be the intersection of the affine cone of 
$\mathrm{Gr}(d,n)$ with the canonical torus 
of $K^{N}$.
Then the \emph{tropical Grassmannian} is the tropicalization $\mathrm{trop}(\mathrm{Gr}_0(d,n))$ which we will refer to as
$\mathrm{TGr}(d,n)$ for short. 
In the following we will assume the tropical Grassmannian $\mathrm{TGr}(d,n)$ 
to be endowed with the fan structure given 
in \cite{SpeStu}, compare Section~\ref{secResults}.

For our approach we will use the following slight modification of the tropical Grassmannian:
Consider the embedding of 
$\mathrm{Gr}(d,n)$
into projective space
given by the map
\begin{equation}\label{equ2}
G \mapsto [p_1, \ldots, p_N, c_1, \ldots, c_s] \in \bP^{N+s-1}.
\end{equation}
Then, intersecting the affine cone over 
$\mathrm{Gr}(d,n)$ with respect to this embedding
with the canonical torus
of $K^{N+s}$ we can form its tropicalization, which we will refer to as $T(c_1, \ldots, c_s)$
in the following.
The image of the generic subspaces 
$G \in \mathrm{Gr}(d,n)$ under the map 
(\ref{equ1})
as above turns out to be the Bieri-Groves set
of this very affine variety.
Even more, the subdivision of the 
image into sets of points that give rise
to the same initial forms 
can be obtained by endowing $T(c_1, \ldots, c_s)$ 
with a sufficiently fine fan structure
and looking at the set of integer points in the relative interior of its maximal cones.
Here and in the following a \emph{cone} is always 
a convex polyhedral cone and a \emph{fan} a polyhedral fan,
where we allow the cones to be non-pointed.
Note that in case that there are no 
non-monomial coefficients 
$c_1, \ldots, c_s$ the above map is the usual Plücker embedding and the tropicalization $T(c_1, \ldots, c_s)$
is the tropical Grassmannian.

The latter case shows up when considering smooth del Pezzo surfaces $X_A$ that arise as blow ups of 
$\bP^2$ in five points in general position. 
It was treated in \cite[Sect. 4]{StXu},
where Sturmfels and Xu show that the collection of all possible initial forms for the generating set of $R_A$ gives rise to a subdivision of the tropical Grassmannian $\mathrm{TGr}(2,5)$ into $2400$ maximal cones. $600$ of these cones come from moneric subspaces and all but $60$ of these are even Khovanskii.
Modulo permuations of the indices of the 
Plücker coordinates, the number of moneric classes 
turns out to be seven and the number of Khovanskii
classes is six.
Now calculating the corresponding cones of the toric degenerations they show that the optimal number 
of linear inequalities for computing the Hilbert function is $12$. 
Moreover, they conjecture that 
in the case of del Pezzo surfaces of degree $3$ 
the optimal number is $21$. 
Classifying all equivalence classes of $3$-dimensional 
Khovanskii subspaces of $K^6$ would enable us to
prove this conjecture.

The first step
of our agenda turns out
to be the task of computing the integer points of a tropical variety in $\bR^{56}$: The embedding of $\mathrm{Gr}(3,6)$ 
as described in (\ref{equ2}) is given by the Plücker coordinates 
and the coefficients of the monomials of the generators $E_i, F_{ij}$ and $G_i$ of $R_A$
that are not monomials in Plücker coordinates. The latter ones are
exactly the $36$ binomial coefficients in the generators $G_i$. 
Unfortunately, our experiments suggest that this approach is computationally infeasible, which leads to the task of finding a computable fan parameterizing moneric and Khovanskii subspaces. 
In \cite{BeCoDBFM} the tropical Grassmannian and the so called Naruki fan are investigated in order to attack this problem 
but turn out to be insufficient for this task:
In their Example 5.3 they give two threedimensional linear subspaces of $K^6$
having the same values in their Plücker coordinates. But whereas the one 
is moneric the other one has a binomial as initial form of one of its generators. 
Similarly they proceed for the Naruki fan.

\section{Idea behind the classification}\label{SecIdea}

From now on we work on the case of smooth del Pezzo surfaces of degree~3.
While in the previous section we claim that $\tgr$ is not suitable for a direct use in the classification of Khovanskii bases of~$\crxa$, here we argue that it still contains a lot of information on these bases and can play a very important role in the solution of this problem.

We compare the structure of $\tgr$ with the structure of the tropical variety $T(c_1, \ldots, c_{36}) \subseteq \bR^{56}$
arising via the embedding of $\gr$ in $\bP^{55}$ as in (\ref{equ2}), 
where the $c_i$ are the binomials appearing in the generators of $\crxa$. 
We set $\tb:= T(c_1, \ldots, c_{36})$ for short and consider the map
\begin{equation}\label{themap}
\tpr \colon \tb \longrightarrow \tgr
\end{equation} 
which comes from a projection to Pl\"ucker coordinates.
Since this projection restricts to a monomial map of big tori, by~\cite[Cor.~3.2.13]{MacLStur} we know that $\tpr$ is well-defined and surjective. Moreover, note that the inverse image of any cone of $\tgr$ under $\tpr$ is a fan, as it is an intersection of a rational polyhedral cone with a fan structure on $\tb$.

The idea behind using $\tpr$ is that since $\tgr$ does not have enough data and $\tb$ does, but is currently uncomputable, we investigate both at the same time. We try to obtain as much of the classification as possible from $\tgr$, but look for help at $\tpr^{-1}(C)$ for all subcones $C$ of $\tgr$ which do not contain the full information on Khovanskii bases of $\crxa$ corresponding to its points. Moreover, we hope that understanding the structure of $\tpr$ may possibly give us some insight into the cardinality of its fibers over such subcones $C$, even without determining $\tpr^{-1}(C)$ explicitly.

\subsection{Plan for the classification}

The reason for the fact that $\tgr$ is insufficient for parametrizing all Khovanskii bases of $\crxa$ is the possibility of a \emph{cancellation} in a binomial. If the value of $\nu$ is the same for both monomials in a binomial $b = m_1 + m_2$, their leading terms may cancel and $\nu(b)$ may be bigger than $\nu(m_1) = \nu(m_2)$.  In particular, $\nu(b)$ cannot be determined by the values $\nu(m_1), \nu(m_2)$ (in which case we could treat $b$ as a monomial for the purpose of determining Khovanskii bases). On the level of Bieri-Groves sets, this is the only possible cause for the presence of non-trivial fibers of $\tpr$ -- either positive dimensional or consisting of several isolated points.

As shown by~\cite[Ex.~5.3]{BeCoDBFM}, $\tgr$ contains a locus where some binomial coefficients of generators $G_1,\ldots, G_6$ may cancel and $\tpr$ actually has a non-trivial fiber. A condition for a binomial having equal values of monomials is a linear form in $\nu(p_1), \ldots, \nu(p_N)$. Hence the locus $L_{BC}$ where a cancellation of at least one binomial coefficient is possible can be easily determined as an intersection of the fan $\tgr$ with a union of finitely many hyperplanes. This construction shows also that $L_{BC}$ has a fan structure. Our situation can be summarized as follows.

\begin{rem}
Over $\tgr \setminus L_{BC}$ the map $\tpr$ is one-to-one. That is, outside $L_{BC}$ the leading terms of elements of~$\mathcal F$ are uniquely determined by $\nu(p_1),\ldots, \nu(p_N)$ and thus are parametrized by integral points of $\tgr$. Over $L_{BC}$ the fibers of $\tpr$ may, but do not have to, be non-trivial.
\end{rem}

Hence we intend to classify Khovanskii bases of $\crxa$ for $G$ corresponding to points of $\tgr \setminus L_{BC}$ using $\tgr$, and investigate fibers of $\tpr$ to find Khovanskii bases for $A$ coming from points of $L_{BC}$. We proceed along the following lines.

\begin{alg}\label{algorithm}
Start with a cone $C$ in $\tgr$, not necessarily maximal. The output is the set of all Khovanskii bases corresponding to points of $\tpr^{-1}(C) \subset \tb$.
\begin{enumerate}
\item \textbf{Using generators corresponding to lines.} Subdivide~$C$ according to the generators $F_{ij}$ of $\crxa$. Cones $C_1,\ldots, C_n$  in the subdivision are in one-to-one correspondence with all possible (moneric) choices of leading terms of all of the generators $F_{ij}$.

Our computations in the proof of Theorem~\ref{thm:SubdivTGr} show that such a choice contains already a lot of information on the relation between values of $\nu$ for the Pl\"ucker coordinates. That is, for each $C_i$ there is not much freedom in choosing leading terms of $G_1,\ldots, G_6$.

\item \textbf{Using generators corresponding to conics.} Subdivide each of the $C_i$ according to choices of leading monomials of $G_1,\ldots, G_6$. While investigating possible values of binomial coefficients one should be able also to describe $L_{BC} \cap C_i$.

\item \textbf{Investigating $\tpr$ over $L_{BC}$.} Take all subcones for which it is not clear whether there is a single choice of leading monomials of $G_1,\ldots,G_6$, and investigate the fibers of $\tpr$ over them. If for a fixed subcone there are only a few binomials that can cancel, one may try to compute the tropicalization of an embedding of $\tgr$ via Pl\"ucker relations and these binomials, and use the result to understand fibers of $\overline{\pi}$. 

We test this approach in Section~\ref{secResults}, where we compute the tropicalization of the embedding of $\gr$ in $\bR^{21}$ given by Pl\"ucker coordinates and one binomial $c$ to prove that over one cone in the subdivision fibers of $\tpr$ are trivial despite the fact that $c$ may cancel.

\item \textbf{Locating Khovanskii bases.} Decide which moneric bases found by the algorithm are 
Khovanskii bases. We deal with this question in Section~\ref{secKhovanskii} using the idea explained in~\cite[Sect.~4,5]{StXu}.

\end{enumerate}
\end{alg}

This general scheme will be modified and extended in order to reduce the complexity of the necessary computations. For the first modifications see Section~\ref{sect:reducing}.

\begin{rem} \label{rem:nonMaxCones}
Note that we have to consider also non-maximal cones of $\tgr$, because they can lie in the image of a maximal cone of~$\tb$. However, most of the faces of cones obtained in the process of subdividing do not have to be analyzed. A face between two cones in the subdivision related to a generator $f$ of~$\crxa$ is, by definition, the locus between regions corresponding to different choices of leading monomials of~$f$. Thus, it consists of points for which two monomials of $f$ have coefficients with the same value, i.e.\ $f$ is not moneric. Only if such a face lies in $L_{BC}$, it has to be taken into consideration.
\end{rem}


\subsection{Reducing the complexity}\label{sect:reducing}

Here we present two important observations which simplify the realization of Algorithm~\ref{algorithm}: reduce the computational complexity and give a better description of~$L_{BC}$. 

First, we consider the action of $\mathfrak S_6$ on $\gr$ coming from permuting columns of the matrix~$A$, or coordinates in the subspace $G$.  Note that a del Pezzo surface $X_A$ is not affected by this action, because permuting columns of $A$ is just permuting the set of points in~$\bP^2$ which are blown up.
This action can be realized as signed permutations of the Pl\"ucker coordinates $\{p_{ijk} \colon 1 \leq i < j < k \leq 6\}$ coming from permutations of the 6 indices, which preserve Pl\"ucker relations.
Thus it extends from the Pl\"ucker embedding of the affine cone over $\gr$ to the linear automorphism of its ambient affine space. This action, in turn, induces an action on $\tgr$ (also on the fan structure), and the tropicalization map is $\mathfrak S_6$-equivariant.

\begin{lemma}\label{lemma:symmetries}
If $P_1, P_2\in \tgr$ are in the same orbit of the $\mathfrak S_6$ action then the integer points in the fibers $\tpr^{-1}(P_1)$ and $\tpr^{-1}(P_2)$ correspond to the same choices of leading terms of generators of~$\crxa$ up to a permutation of variables. 
\end{lemma}

\begin{proof}
The generating set $\mathcal{F} \subset  \mathbb{K}[p_{ijk} \colon 1\leq i < j < k \leq 6][x_1,\ldots,x_6,y_1,\ldots, y_6]$ is invariant (up to signs and Pl\"ucker relations) under the $\mathfrak S_6$-action. (This can be checked either by direct computation or, more geometrically, by referring to the Cox ring structure.)
We conclude that $\mathfrak S_6$ acts also on the set of binomial coefficients of generators $G_1,\ldots, G_6$ up to signs. This allows us to define the $\mathfrak S_6$-action on $\tb$. Note that $\tpr$ becomes an equivariant map, which ends the proof.
\end{proof}

\begin{cor}
We may restrict to applying Algorithm~\ref{algorithm} to representatives of $\mathfrak S_6$-orbits in~$\tgr$.
\end{cor}

Using Lemma~\ref{lemma:symmetries} reduces the number of cones to consider significantly. For instance, instead of processing all 1005 maximal dimensional cones of $\tgr$ we have to deal with only~7 representatives of equivalence classes, see Section~\ref{secResults}.

The second observation concerns \emph{Pl\"ucker-equivalent} polynomials: we say that $f_1, f_2 \in \mathbb{K}[p_{ijk} \colon 1\leq i < j < k \leq 6]$ are Pl\"ucker-equivalent if their difference belongs to the Pl\"ucker ideal. 
Note that if $f_1, f_2$ are Pl\"ucker-equivalent  then we have $\nu(f_1) = \nu(f_2)$ on $\gr$, because by substituting all $p_{ijk}(G) \in \mathbb{K}(t)$ for any $G \in \gr$ into $f_1-f_2$ we get~0. This proves the following

\begin{lemma}\label{lem:equivbin}
Instead of considering a single binomial we may look at a whole class of Pl\"ucker-equivalent binomials (or even polynomials with more terms). If for some $G \in \gr$ one of them may cancel, i.e.\ values of both monomials are the same, we look for an equivalent one which cannot cancel at $G$. If we find such a binomial, we can read out its value only from values of Pl\"ucker coordinates, without passing to $\mathbb{K}(t)$.
\end{lemma}

The computations done so far (see Section~\ref{secResults}) show that thanks to Lemma~\ref{lem:equivbin} we get significantly smaller $L_{BC}$ and therefore less lower-dimensional cones to process. For every binomial coefficient of $G_i$ one can find several other binomials equivalent to it up to Pl\"ucker relations. It happens very often that even if one binomial in a group of equivalent ones cancels there is one which does not.

\section{Classification of moneric subspaces from \texorpdfstring{$\tgr \setminus L_{BC}$}{TGr(3,6) minus L\_BC}}\label{secResults}

As explained in the previous section, the map $\tpr \colon \tb \to \tgr$ is bijective outside the locus $L_{BC} \subseteq \tgr$ given by possible cancellations of lowest order terms in some binomial expressions. The locus where $\tpr$ is not one-to-one induces subdivisions of the cones in $\tgr$ and the resulting fan structure can in turn can be refined further by choosing initial terms for each of the generators in $\mathcal F$. This refined fan structure on $\tgr$ has the property that each maximal cone gives rise to a class of moneric subspaces, while the non-maximal cones contained in $L_{BC}$ require further investigation, see Remark~\ref{rem:nonMaxCones}.

This fan structure is combinatorially extremely large and carrying out the corresponding subdivisions of cones in $\tgr$ is computationally infeasible. However, the following observation helps: Often, when a cancellation in the lowest order terms of a binomial occurs, the leading monomial of the corresponding generator in $\mathcal F$ is independent of the exact value of the binomial because there is another term of lower value. Systematically exploiting this observation leads to a unique coarsening of the described fan structure. We were able to compute this coarsened fan structure and use it to classify all moneric subspaces $G$ whose corresponding tropical point $\nu(G) := (\nu(p_{123}), \ldots, \nu(p_{456}))$ in $\tgr$ does not lie in $L_{BC}$:

\begin{thm} \label{thm:SubdivTGr}
  There is a unique coarsest refinement $\Sigma$ of the fan structure on $\tgr$ equipped with linear maps $\theta_C \colon C \to \bR^{27}$ for all maximal cones $C$ in $\Sigma$ such that every subspace $G \in \grO$ tropicalizing to the relative interior of $C$ is moneric of weight $\theta_C(\nu(G))$. Its $f$-vector is
  \begingroup \small
  \[(0,\,1,\,987,\,25605,\,245280,\,1195815,\,3380380,\,5827950,\,6076590,\,3524580,\,870840)\]
  \endgroup
  and the maximal cones describe $32880$ distinct moneric classes, which fall into $78$ orbits under the $\mathfrak S_6$-symmetry.
\end{thm}

\begin{proof}
  We denote by $\rpl := \bK[p_{123},\ldots,p_{456}]/\pl$ the homogeneous coordinate ring of $\gr$ under its Plücker embedding. 
  For any expression $c \in \rpl$, we consider its \emph{tropical graph}  $T(c) \subseteq \bR^{21}$, defined as the tropicalization of 
  the embedding of $\gr$ into $K^{21}$ via its Plücker coordinates together with expression $c$.
  Moreover, by $\trop(c)$ we denote the piece-wise linear function $\tgr \to \bR$, mapping $(w_{123},\ldots, w_{456})$ to $\min\{\sum_{ijk} \alpha_{ijk}^\supind{\ell} w_{ijk} \mid \ell = 1, \ldots, r\}$, where $\alpha^\supind{1}, \ldots, \alpha^\supind{r}$ are the exponents in $c$.
  
  For convenience, we number the expressions for the generators of the Cox-Nagata ring as $\mathcal F = \{f_1, \ldots, f_{27}\} \subseteq \rpl[x_1,y_1,\ldots, x_6, y_6]$. For each $i = 1,\ldots,27$, let $c_{i1},\ldots,c_{im_i} \in \rpl$ denote the coefficients of $f_i$ with respect to the variables $x_1,y_1,\ldots,x_6,y_6$. Define $\Lambda_i$ as the lower convex hull of $T(c_{i1}) \cup \ldots \cup T(c_{im_i})$ in $\bR^{21}$ and let $\Sigma_i$ be the subdivision of $\tgr$ induced from $\Lambda_i$ under the projection $\pi \colon \bR^{21} \twoheadrightarrow \bR^{20}$.
  Consider the fan structure $\Sigma$ on $\tgr$ arising as the common refinement of the subdivisions $\Sigma_1,\ldots,\Sigma_{27}$.
  
  Let $C$ be any maximal cone of $\Sigma$. By construction, for $i=1,\ldots,27$, the lower convex hull $\Lambda_i$ is linear over the relative interior of $C$, i.e.,
  \[\Relint(\Lambda_i \cap \pi^{-1}(C)) = \{(w,\theta_i(w)) \in \bR^{21} \mid w \in \Relint(C)\}\]
  for a uniquely determined linear map $\theta_i \colon C \to \bR$. Moreover, the explicit computation of $\Lambda_i$ (as described below) reveals that the relative interior of any maximal cone of $\Lambda_i$ intersects only one of the tropical graphs $T(c_{i1}), \ldots, T(c_{im_i})$. In particular, $\Relint(\Lambda_i \cap \pi^{-1}(C))$ is contained in $T(c_{ij})$ for exactly one $j$ and does not intersect $T(c_{ik})$ for $k \neq j$. This shows that for all $G \in 
  \grO$ tropicalizing to $w \in \Relint(C)$, the $j$-th coefficient $c_{ij}$ is the unique term of smallest value $\theta_i(w)$ among all terms of $f_i$. Defining $\theta_C := (\theta_1,\ldots,\theta_{27}) \colon C \to \bR^{27}$, we conclude that every $G$ with $\nu(G) \in \Relint(C)$ gives rise to a set $R_A$ that is moneric of weight $\theta_C(w)$.
  
  The direct calculation of $\Sigma$ is a computationally difficult task, as evidenced by its large $f$-vector. However, by exploiting symmetries at several stages and making use of parallel computations, we were able to carry out the computation with the computer algebra system \textsc{Singular} \cite{Singular} using its interface to \textsc{Gfan} \cite{gfan} and its library for computing tropical varieties \cite{tropicallib}. 
  
  We proceed as follows: We determine $\Sigma$ by computing for every single maximal cone in $\tgr$ its subdivision in $\Sigma$. Up to $\mathfrak S_6$-symmetry, $\tgr$ consists only of seven maximal cones \cite{SpeStu}. These seven cones themselves are stabilized under a subgroup of $\mathfrak S_6$ of order $48$, $24$, $8$, $8$, $4$, $3$ and $2$, respectively. In the following steps of the computation, we fix one of these seven maximal cones $C$ and its symmetry group $\mathfrak S(C) \subseteq \mathfrak S_6$.
  
  Most of the coefficients $c_{ij} \in \rpl$ in $f_1,\ldots,f_{27}$ are monomial expressions in the Plücker coordinates. For these $c_{ij}$, the tropical graph $T(c_{ij})$ is simply the image of $\tgr$ under the linear embedding $\bR^{20} \hookrightarrow \bR^{21}$, $w \mapsto (w,\, \trop(c_{ij})(w))$. 
  
  The remaining coefficients can be expressed as $c_{ij} = u-v$ with $u,v$ monomials in $\bK[p_{123},\ldots,p_{456}]$. If the linear functions $\trop(u),\, \trop(v) \colon \bR^{20} \to \bR$ do not coincide on the entire cone $C$ for at least one of the equivalent binomial expressions for $c_{ij}$ (see Lemma~\ref{lem:equivbin}), then the lower convex hull of $T(c_{ij}) \cap \pi^{-1}(C)$ is the graph of the piece-wise linear function $\restr{\trop(u-v)}{C} \colon C \to \bR$. This is in fact always the case except for one special situation: For one of the seven cones $C$, there exists one coefficient $c_{ij}$ for which all known expressions as binomials $u-v$ satisfy $\restr{\trop(u)}{C} = \restr{\trop(v)}{C}$. For this one choice of $C$, $c_{ij}$, we computed the tropical variety $T(c_{ij})$ explicitly in \textsc{Singular} and read off that the lower convex hull of $T(c_{ij}) \cap \pi^{-1}(C)$ is in fact the graph of a linear function $C \to \bR$.
  
  From these descriptions of the lower convex hulls of $T(c_{ij}) \cap \pi^{-1}(C)$ for all $c_{ij}$, we compute $\Lambda_i \cap \pi^{-1}(C)$ for each $i=1,\ldots,27$ and the induced subdivisions $\restr{\Sigma_i}{C}$ of $C$. Our next step is to compute the common refinement of $\restr{\Sigma_1}{C}, \ldots, \restr{\Sigma_{27}}{C}$, which is combinatorially large, so we find it imperative to exploit symmetries and use parallelization in the computation, as described in the following algorithm. Here, each subdivision $\restr{\Sigma_i}{C}$ is understood as a set of maximal cones.
  
  \begingroup \small
  \begin{algorithmic}[1]
    \REQUIRE{The subdivisions $\restr{\Sigma_1}{C}$, \ldots, $\restr{\Sigma_{27}}{C}$ and the symmetry group $\mathfrak S(C)$}
    \ENSURE{$\restr{\Sigma}{C}$, the common refinement of $\restr{\Sigma_1}{C},\ldots,\restr{\Sigma_{27}}{C}$}
    \STATE $\Omega := \{C\}$.
    \FOR{$i\in\{1,\ldots,27\}$}
    \FORALL[\AlgComment{3em}{To be carried out in parallel}]{$\sigma \in \restr{\Sigma_i}{C}$}
    \STATE Compute $\Omega_\sigma := \{\sigma \cap \sigma' \mid \sigma' \in \Omega \text{ such that } \dim(\sigma \cap \sigma') = \dim(\sigma)\}$.
    \ENDFOR
    \STATE $\Omega := \{\tau_1,\ldots,\tau_k\}$, a set of 
    representatives for the $\mathfrak S(C)$-orbits of $\bigcup_{\sigma \in \restr{\Sigma_i}{C}} \Omega_\sigma$.
    \ENDFOR
    \RETURN{$\bigcup_{g \in \mathfrak S(C)} \{g \cdot \sigma \mid \sigma \in \Omega\}$.}
  \end{algorithmic}
  \endgroup
  
  With this common refinement algorithm, we were able to compute the maximal cones in the subdivision $\restr{\Sigma_i}{C}$ for each of the seven maximal cones $C$ representing the $\mathfrak S_6$-symmetry classes in $\tgr$. After letting the symmetric group $\mathfrak S_6$ act on these results and computing the fan structure induced by the set of maximal cones, we obtain $\Sigma$ and read off the $f$-vector claimed above.
  
  During the computation, we remember at every stage for each maximal cone the choice of initial monomials it corresponds to. At the end, the computation reveals that the $870840$ maximal cones only describe $32880$ distinct moneric classes, and they fall into only $78$ orbits under the $\mathfrak S_6$-symmetry.  More detailed statistics on the subdivisions for each of the seven symmetry classes of cones are presented in Figure~\ref{fig:coneStatistics} and an explicit description of the subdivision for one of the maximal cones in $\tgr$ is contained among the discussion in Section~\ref{sec:EEEE}.
\end{proof}

\begin{figure}[h]
  \begingroup \footnotesize
  \setlength{\tabcolsep}{0.32em}
  \begin{tabular}{l|ccccccc}
    Type of cone $C$ & 
    FFFGG & EEEE & EEFF1 & EEFF2 & EFFG & EEEG & EEFG \\[0.1em] \hline \\[-0.8em]
    size of $\mathfrak S_6$-orbit $\mathfrak S_6 \cdot C$ &
    15 & 30 & 90 & 90 & 180 & 240 & 360 \\
    order of symmetry group $\mathfrak S(C)$ &
    48 & 24 & 8 & 8 & 4 & 3 & 2 \\
    $\#\{\text{cones in }\restr{\Sigma}{C}\}$ &
    1240 & 864 & 1248 & 860 & 806 & 830 & 812 \\
    $\#\{\text{cones in }\restr{\Sigma}{C}\}/\mathfrak S(C)$ &
    38 & 36 & 205 & 142 & 259 & 278 & 460 \\
    $\#\{\text{moneric classes from } C\}/\mathfrak S(C)$ &
    31 & 35 & 162 & 135 & 253 & 274 & 449 \\
    $\#\{\text{moneric classes from }\mathfrak S_6 \cdot C\}/\mathfrak S_6$ & 
    25 & 17 & 45 & 47 & 60 & 37 & 64 \\
    $\#\{\text{moneric classes from }\mathfrak S_6 \cdot C\}$ & 
    11040 & 7080 & 16080 & 17880 & 24600 & 17040 & 26400
  \end{tabular}
  \endgroup
  \caption{Statistics on moneric classes arising from $\Sigma$}
  \label{fig:coneStatistics}
\end{figure}

\section{A complete classification of monericity in EEEE-cones} \label{sec:EEEE}

In the previous section, we computed all moneric classes arising from maximal cones of $\tb$ mapping injectively to $\tgr$ under $\tpr$. On the other hand, maximal cones of $\tb$ not mapping full-dimensionally to $\tgr$ may also give rise to moneric subspaces. To classify also such moneric classes, a more refined approach is necessary. Here, we present partial results in this direction.

Recall that $\tgr$ consists of $1005$ maximal cones which fall into seven orbits under the $\mathfrak S_6$-action, see \cite{SpeStu}. In the following, we focus our attention to one of these seven symmetry classes, called EEEE in \cite{SpeStu}. Going beyond Theorem~\ref{thm:SubdivTGr}, we study all moneric subspaces arising from the relative interior of cones of this symmetry class, also classifying the behaviour along $L_{BC}$. We suspect that a similar approach would be useful for treating all cones in $\tgr$.

By \cite{SpeStu}, a cone $C$ representing the symmetry class EEEE is given by the image of the cone $\bR^6 \times \bR_{\geq 0}^4$ under the linear map
$\varphi \colon \bR^6 \times \bR^4 \hookrightarrow \bR^{20}$,
\[\varphi(a,b) := \sum_{i<j<k} (a_{i}+a_{j}+a_{k})e_{ijk} + b_1 e_{123} + b_2 e_{145} + b_3 e_{246} + b_4 e_{356} \in \bR^{20}.\]
Its symmetry group $\mathfrak S(C) \subseteq \mathfrak S_6$ is the subgroup of $\mathfrak S_6$ of order $24$ stabilizing
${(e_1 \wedge e_6) \cdot (e_2 \wedge e_5) \cdot (e_3 \wedge e_4)} \in \Sym^3 \bigwedge^2 \bR^6$.
In the following, let $G \in \grO$ be a moneric subspace such that its corresponding point in $\tgr$ is $\nu(G) = \varphi((a,b))$ with $(a,b) \in \bR^6 \times \bR_{>0}^4$.

First, we observe that the coefficients in $F_{16}$ are $p_{126}, p_{136}, p_{146}$ and $p_{156}$, so their values are $\{a_1+a_6+a_i \mid i = 2, 3, 4, 5\}$. Since $\init(F_{16})$ is a monomial, this implies that there is a unique smallest number among $\{a_2,a_3,a_4,a_5\}$. In the same way, considering the coefficients of $F_{25}$ and $F_{34}$ reveals that both $\{a_1,a_3,a_4,a_6\}$ and $\{a_1,a_2,a_5,a_6\}$ have a unique smallest element. Up to $\mathfrak S(C)$-symmetry, 
we may assume that $a_1 < a_i$ for $i = 2,\ldots,5$ and that $a_2 < a_j$ for $j = 3,4,5$. The coefficients of $F_{14}$ have values $a_1 + a_4 + \{a_2, a_3, a_5+b_2, a_6\}$, so $a_2 \neq a_6$ holds, as otherwise $\init(F_{14})$ would not be a monomial. This leads to two cases to consider: $a_6 < a_2$ and $a_2 < a_6$.
\medskip

\noindent \emph{Case 1}: $a_1 < a_6 < a_2 < a_i$ for all $i = 3, 4, 5$.

Considering the expression \eqref{eq:genF} for $F_{ij}$, we observe that the four coefficients of $F_{ij}$ have values at least $a_i+a_j+a_k$, where $k$ ranges over $\{1,\ldots,6\} \setminus \{i,j\}$, with equality if and only if $\{i,j,k\} \neq \{1,2,3\},\{1,4,5\},\{2,4,6\},\{3,5,6\}$. In particular, the above inequalities for $a_1, \ldots, a_6$ uniquely determine
\begin{gather*}
\init(F_{ij}) = \frac{x_2 x_3 x_4 x_5 x_6}{x_i x_j} y_1 \qquad \text{for } 1 < i < j,\, (i,j) \neq (2,3), (4,5), \\
\init(F_{1i}) = \frac{x_2 x_3 x_4 x_5}{x_i} y_6 \qquad \text{for }i=2,3,4,5 \qquad \text{and} \qquad \init(F_{16}) = x_3 x_4 x_5 y_2.
\end{gather*}
For $F_{23}$ and $F_{45}$, the coefficients have values $a_2+a_3+\{a_1+b_1,a_4,a_5,a_6\}$ and $a_4+a_5+\{a_1+b_2,a_2,a_3,a_6\}$, respectively, leading to the following four possibilities for the initial monomials $\init(F_{23})$, $\init(F_{45})$:

\begingroup \small
\begin{center}
  \setlength{\tabcolsep}{1em}
  \begin{tabular}{c|cc}
    & $b_1 < a_6-a_1$ & $b_1 > a_6 - a_1$ \\ \hline
    $b_2 < a_2-a_1$ & $x_4 x_5 x_6 y_1$, $x_2 x_3 x_6 y_1$ & $x_1 x_4 x_5 y_6$, $x_2 x_3 x_6 y_1$ \\
    $b_2 > a_2-a_1$ & $x_4 x_5 x_6 y_1$, $x_1 x_2 x_3 y_6$ & $x_1 x_4 x_5 y_6$, $x_1 x_2 x_3 y_6$.
  \end{tabular}
\end{center}
\endgroup


In the expression \eqref{eq:genG} for $G_i$, the monomials are of the form  $\frac{x_1 \cdots x_6 x_i}{x_j x_k} y_j y_k$, where $(j,k)$ ranges over tuples with $j \neq k$ or $j = k = i$. Examining the coefficient of such a monomial, we observe that it has value at least $2\sum_{\ell \neq i}a_\ell + a_j + a_k$. From this, we see that $\init(G_i) = x_2 x_3 x_4 x_5 x_i y_1 y_6$ for $i=2,3,4,5$,
as the corresponding coefficients exactly attain this bound, while the inequalities among $a_1, \ldots, a_6$ force the remaining coefficients of $G_i$ to have higher value.

For $G_1$, we observe that the coefficient of $y_1^2 x_2 x_3 x_4 x_5 x_6$ is a binomial $u-v$ in the Plücker coordinates, where 
$\nu(v) = \nu(u)+b_3$.
In particular, $\nu(u-v) = \nu(u) = 2\sum_i a_i$, which is smaller than the values of all other coefficients of $G_1$. Hence, $\init(G_1) = x_2 x_3 x_4 x_5 x_6 y_1^2$. It remains to consider 
\[G_6 = (u_1-u_2) x_2 x_3 x_4 x_5 x_6 y_1 y_6 + u_3 x_3 x_4 x_5 x_6^2 y_1 y_2 + u_4 x_1 x_2 x_3 x_4 x_5 y_6^2 + \ldots,\]
where $\{u_1, u_2, u_3, u_4\}$ are monomials in the Plücker coordinates of values
\[2a_1+\ldots+2a_5 + \{a_1 + a_6 + b_1,\, a_1 + a_6 + b_2,\, a_1 + a_2 + b_1,\, 2a_6\}\]
and the remaining terms of $G_6$ are each of higher value than at least one of them. Refining the conditions imposed by the four possibilities for $\init(F_{23}),\init(F_{45})$ leads to the following possible leading monomials for $G_6$:

\begingroup\footnotesize
\begin{center}
  \begin{tabular}{ccc}
    conditions on $a \in \bR^6$, $b\in \bR_{>0}^4$ & values of $u_1, \ldots, u_4$ & $\init(G_6)$ \\[0.4em] 
    $b_1 > a_6-a_1$, $b_2 > a_2-a_1$ & $\nu(u_4) < \nu(u_1),\nu(u_2),\nu(u_3)$ & $x_1 x_2 x_3 x_4 x_5 y_6^2$ \\
    $b_1 > a_6-a_1$, $b_2 < a_2-a_1$ & $\nu(u_2) < \nu(u_4) < \nu(u_1) < \nu(u_3)$ & $x_2 x_3 x_4 x_5 x_6 y_1 y_6$ \\
    $b_1 < a_6-a_1$, $b_2 > a_2-a_1$ & $\nu(u_1) < \nu(u_2), \nu(u_3), \nu(u_4)$ & $x_2 x_3 x_4 x_5 x_6 y_1 y_6$ \\
    $b_1 < a_6-a_1$, $b_2 < a_2-a_1$, $b_1 < b_2$ & $\nu(u_1) < \nu(u_2), \nu(u_3),\nu(u_4)$ & $x_2 x_3 x_4 x_5 x_6 y_1 y_6$ \\
    $b_2 < b_1 < a_6-a_1$ 
    & $\nu(u_2) < \nu(u_1) < \nu(u_3),\nu(u_4)$ & $x_2 x_3 x_4 x_5 x_6 y_1 y_6$ \\
    $b_1 = b_2 < a_6-a_1$ 
    & $\nu(u_1) = \nu(u_2) < \nu(u_3),\nu(u_4)$ & ? \\
  \end{tabular}
\end{center}
\endgroup

In the last case $b_1 = b_2 < a_6 - a_1$, the binomial $u_1-u_2$ can attain any value $\lambda \geq b_1 + 2\sum_{i=1}^6 a_i - (a_6-a_1)$. Indeed, the tropical graph $T(u_1-u_2) \subseteq \bR^{21}$ 
contains all points of the form $(\varphi(a',b'),\lambda')$ with $a' \in \bR^6$, $b' \in \bR_{>0}^4$, $b'_1 = b'_2$ and $\lambda \geq b'_1 + 2\sum_{i=1}^6 a_i - (a_6-a_1)$.
In particular, the following three cases occur under the specified condition on $\lambda_0 := \nu(u_1-u_2)-2\sum_{i=1}^6 a_i + (a_6-a_1) - b_1$:

\begingroup\footnotesize
\begin{center}
  \begin{tabular}{cc}
    conditions on $(a,b,\lambda_0) \in \bR^6 \times \bR_{>0}^4 \times \bR_{>0}$ & $\init(G_6)$ \\[0.4em] 
    $b_1 = b_2 < a_6-a_1$, $\lambda_0 < a_2-a_6$, $\lambda_0 < (a_6-a_1)-b_1$ & $x_2 x_3 x_4 x_5 x_6 y_1 y_6$ \\
    $b_1 = b_2 < (a_6-a_1)-(a_2-a_6)$, $\lambda_0 > a_2 - a_6$ & $x_3 x_4 x_5 x_6^2 y_1 y_2$ \\
    $(a_6-a_1)-(a_2-a_6) < b_1 = b_2 < a_6-a_1$, $\lambda_0 > (a_6-a_1)-b_1$ & $x_1 x_2 x_3 x_4 x_5 y_6^2$ \\
  \end{tabular}
\end{center}
\endgroup
\medskip

\noindent \emph{Case 2}: $a_1 < a_2 < a_i$ for all $i = 3, 4, 5, 6$.

We proceed as in the previous case, examining the restrictions on the leading terms of $\mathcal F$ imposed by the inequalities $a_1 < a_2 < a_3,a_4,a_5,a_6$. We always have
\begingroup \small
\begin{align*}
   \init(F_{1j}) &= \frac{x_3 x_4 x_5 x_6}{x_j} y_2,  \quad \init(F_{ij}) = \frac{x_2 x_3 x_4 x_5 x_6}{x_i x_j} y_1 \quad \text{for }j \geq 4, i \neq 1,6, (i,j) \neq (4,5), \\
   \init(G_1) &= x_2 x_3 x_4 x_5 x_6 y_1^2, \quad \init(G_2) = x_2 x_3 x_4 x_5 x_6 y_1 y_2, \quad \init(G_3) = x_3^2 x_4 x_5 x_6 y_1 y_2,
\end{align*}
\endgroup
while for the remaining generators, we obtain
%
\begingroup \footnotesize
\begin{align*}
&\init(F_{45}) \in \{x_2 x_3 x_6 y_1,\, x_1 x_3 x_6 y_2\} =: X_1, \\
&\init(F_{12}) \in \{x_4 x_5 x_6 y_3,\, x_3 x_5 x_6 y_4,\, x_3 x_4 x_6 y_5,\, x_3 x_4 x_5 y_6\} =: X_2, \\
&\init(F_{13}) \in \{x_4 x_5 x_6 y_2,\, x_2 x_5 x_6 y_4,\, x_2 x_4 x_6 y_5,\, x_2 x_4 x_5 y_6\} =: X_3, \\
&\init(F_{23}) \in \{x_4 x_5 x_6 y_1,\, x_1 x_5 x_6 y_4,\, x_1 x_4 x_6 y_5,\, x_1 x_4 x_5 y_6\} =: X_4, \\
&\init(G_5) \in \{x_3 x_4 x_5^2 x_6 y_1 y_2,\, x_2 x_3 x_5^2 x_6 y_1 y_4,\, x_2 x_3 x_4 x_5 x_6 y_1 y_5,\, x_2 x_3 x_4 x_5^2 y_1 y_6\} := X_5, \\
&\init(G_4) \in \{x_3 x_4^2 x_5 x_6 y_1 y_2,\, x_2 x_4^2 x_5 x_6 y_1 y_3,\, x_2 x_3 x_4 x_5 x_6 y_1 y_4,\, x_2 x_3 x_4^2 x_6 y_1 y_5,\, x_2 x_3 x_4^2 x_5 y_1 y_6\} =: X_6, \\
&\init(G_6) \in \{x_3 x_4 x_5 x_6^2 y_1 y_2,\, x_2 x_3 x_5 x_6^2 y_1 y_4,\, x_1 x_3 x_5 x_6^2 y_2 y_4,\, x_2 x_3 x_4 x_6^2 y_1 y_5,\,x_1 x_3 x_4 x_6^2 y_2 y_5, \\
&\phantom{\init(G_6) \in \{~} x_2 x_3 x_4 x_5 x_6 y_1 y_6,\, x_1 x_3 x_4 x_5 x_6 y_2 y_6\} =: X_7.
\end{align*}
\endgroup

Opposed to Case 1, one also observes that here no cancellation of lowest order terms in binomials can affect the initial term of any generator. In particular, the initials of the generators $\mathcal F$ only depend on $a$ and $b$, leading to a distinction of $31$ cases of moneric choices, visualized in Figure~\ref{fig:caseDistinction}. There, the choice of initial monomials from the sets $X_1, \ldots, X_7$ is specified by a $7$-tuple $(i_1, \ldots, i_7)$ indicating that the $i_j$-th element of $X_j$ forms the leading monomial.


\begin{figure}[h]
  \begingroup \scriptsize
  \setlength{\tabcolsep}{0.5em}
  \renewcommand{\arraystretch}{1.3}
  \begin{tabular}{r|c|c|c|c|}
    \multicolumn{1}{c}{} & \multicolumn{2}{c}{$b_2 < a_2 - a_1$} & \multicolumn{2}{c}{$b_2 > a_2 - a_1$} \\ \cline{2-5}
    $a_4 < a_5, a_6$, & \multicolumn{2}{|c|}{\multirow{2}{*}{\texttt{2222232}}} & \multicolumn{2}{|c|}{\multirow{2}{*}{\texttt{1222233}}} \\[-0.3em]
    $a_4 < b_1+a_1$ & \multicolumn{2}{|c|}{} & \multicolumn{2}{|c|}{} \\ \cline{2-5}
    $a_5 < a_4, a_6$, & \multicolumn{2}{|c|}{\multirow{2}{*}{\texttt{2333344}}} & \multicolumn{2}{|c|}{\multirow{2}{*}{\texttt{1333345}}} \\[-0.3em]
    $a_5 < b_1+a_1$ & \multicolumn{2}{|c|}{} & \multicolumn{2}{|c|}{} \\ \cline{2-5}
    $a_6 < a_4, a_5$, & \multicolumn{2}{|c|}{\multirow{2}{*}{\texttt{2444456}}} & \multicolumn{2}{|c|}{\multirow{2}{*}{\texttt{1444457}}} \\[-0.3em]
    $a_6 < b_1+a_1$ & \multicolumn{2}{|c|}{} & \multicolumn{2}{|c|}{} \\ \cline{2-5}
    $a_4 < a_5, a_6$, & $b_2 < (a_2+b_1)-a_4$ & $b_2 > (a_2+b_1)-a_4$ & \multicolumn{2}{|c|}{\multirow{2}{*}{\texttt{2221231}}} \\[-0.3em]
    $a_4 \in b_1+(a_1, a_2)$ & \texttt{1221232} & \texttt{1221231} & \multicolumn{2}{|c|}{} \\ \cline{2-5}
    $a_5 < a_4, a_6$, & $b_2 < (a_2+b_1)-a_5$ & $b_2 > (a_2+b_1)-a_5$ & \multicolumn{2}{|c|}{\multirow{2}{*}{\texttt{2331341}}} \\[-0.3em]
    $a_5 \in b_1+(a_1, a_2)$ & \texttt{1331344} & \texttt{1331341} & \multicolumn{2}{|c|}{} \\ \cline{2-5}
    $a_6 < a_4, a_5$, & $b_2 < (a_2+b_1)-a_6$ & $b_2 > (a_2+b_1)-a_6$ & \multicolumn{2}{|c|}{\multirow{2}{*}{\texttt{2441451}}} \\[-0.3em]
    $a_6 \in b_1+(a_1, a_2)$ & \texttt{1441456} & \texttt{1441451} & \multicolumn{2}{|c|}{} \\ \cline{2-5}
    $a_4 < a_5, a_6$, & $b_4 < a_4-(a_2+b_1)$ & $b_4 > a_4-(a_2+b_1)$ & $b_4 < a_4-(a_2+b_1)$ & $b_4 > a_4-(a_2+b_1)$ \\[-0.3em]
    $a_4 \in b_1+(a_2, a_3)$ & \texttt{1211111} & \texttt{1211131} & \texttt{2211111} & \texttt{2211131} \\ \cline{2-5}
    $a_5 < a_4, a_6$, & $b_4 < a_5-(a_2+b_1)$ & $b_4 > a_5-(a_2+b_1)$ & $b_4 < a_5-(a_2+b_1)$ & $b_4 > a_5-(a_2+b_1)$ \\[-0.3em]
    $a_5 \in b_1+(a_2, a_3)$ & \texttt{1311111} & \texttt{1311141} & \texttt{2311111} & \texttt{2311141} \\ \cline{2-5}
    $a_6 < a_4, a_5$, & $b_4 < a_6-(a_2+b_1)$ & $b_4 > a_6-(a_2+b_1)$ & $b_4 < a_6-(a_2+b_1)$ & $b_4 > a_6-(a_2+b_1)$ \\[-0.3em]
    $a_6 \in b_1+(a_2, a_3)$ & \texttt{1411111} & \texttt{1411151} & \texttt{2411111} & \texttt{2411151} \\ \cline{2-5}
    \multirow{2}{*}{$a_4,a_5,a_6 > b_1+a_3$} & $b_4 < a_3-a_2$ & $b_4 > a_3-a_2$ & $b_4 < a_3-a_2$ & $b_4 > a_3-a_2$ \\[-0.3em]
    & \texttt{1111111} & \texttt{1111121} & \texttt{2111111} & \texttt{2111121} \\ \cline{2-5}
  \end{tabular}
  \endgroup
  \caption{Classification of monerics for $a_1 < a_2 < a_3, a_4, a_5, a_6$.}
  \label{fig:caseDistinction}
\end{figure}

In total, Cases 1 and 2 combined, we obtain $39$ sets of inequalities classifying all moneric subspaces inside $C$ up to $\mathfrak S(C)$-symmetry. By considering their $\mathfrak S_6$-orbits, we obtain:

\begin{thm}\label{thm:mainthm5}
  There are exactly $7320$ equivalence classes of moneric subspaces $G \in \grO$ tropicalizing to the relative interior of a maximal cone of type EEEE in $\tgr$. They fall into $19$ orbits under the action of $\mathfrak S_6$.
\end{thm}

\section{Khovanskii or not?}\label{secKhovanskii}


The last step of Algorithm~\ref{algorithm} is to check whether a moneric subspace $G$ is Khovanskii, i.e. whether the set of initial monomials $\{in(f) \colon f \in \mathcal{F}\}$ generates the initial algebra $in(R_A)$ of the ring $R_A$ generated by $\mathcal{F}$. We apply the method presented in~\cite[Thm~5.1]{StXu}, relying on the lifting criterion~\cite[Prop.~1.3]{CoHeVa}. For a set of binomial generators $b_1,\ldots,b_k$ of the ideal $J$ of relations between $\{in(f) \colon f \in \mathcal{F}\}$ we have to check whether each $b_i$ lifts to the ideal $I$ of relations between $\mathcal{F}$. We say that $w(z_1,\ldots,z_m) \in J$ lifts to $I$ if we may write $0 = w(f_1,\ldots,f_m) + S(f_1,\ldots,f_m)$, where $S$ consists of monomials in $f_1,\ldots,f_m$ of value higher or equal than the value of $w(f_1,\ldots,f_m)$.

This criterion can be verified by comparing the dimensions of graded pieces of~$J$ containing a minimal binomial generating set $b_1,\ldots,b_k$ with corresponding graded pieces of~$I$. We define a map $q$ sending a homogeneous element $\overline{w} \in I$ to an element $w \in J$ obtained by taking only monomials with smallest value from~$\overline{w}$ -- the left inverse to the lifting of relations. Thus a binomial $b_i$ can be lifted to $I$ if and only if it lies in the image of $q$. Note that if a graded piece of $I$ has dimension $d$ (as a $K$-vector space) then it is mapped by $q$ to a $d$-dimensional subspace of the corresponding graded piece of $J$.

We consider 78 equivalence classes of moneric subspaces $G$ from Theorem~\ref{thm:SubdivTGr}. In each of these cases minimal generators of $J$ are placed in 27 degrees (the same as in~\cite[Thm~5.1]{StXu}) and one may check that in these degrees there cannot be any more linearly independent generators. For each of these degrees the dimension of the graded piece of $I$ is~3 (in general, one may obtain this dimension from the value of the Hilbert function, computed e.g. from formula~\cite[Cor.~5.2]{StXu}). Hence an equivalence class as above is Khovanskii if and only if all corresponding graded pieces of $J$ contain 3 minimal generators.

\begin{thm}\label{thmKhovanskii}
There are 38 Khovanskii classes in the set of 78 moneric equivalence classes found in Theorem~\ref{thm:SubdivTGr}. In particular, we obtain 38 different combinatorial types of toric degenerations of total coordinate spaces of Cox rings of smooth cubic del Pezzo surfaces.
\end{thm}

\begin{cor}
The smallest number of faces of the cone corresponding to the toric algebra for one of the 38 Khovanskii classes found above is~21. That is, those classes do not give a more compact formula for the Hilbert function of $R_A$ than~\cite[Cor.~5.2]{StXu}.
\end{cor}

\section{Summary}

In Sections~\ref{secResults}-\ref{secKhovanskii} we have shown how one can deal with each step of Algorithm~\ref{algorithm}, and presented examples of application of these ideas. Although the computations are not yet finished, we expect that our methods will be sufficient to understand the missing cases and complete the classification of Khovanskii bases of Cox rings of the smooth cubic del Pezzo surfaces. What we still need to analyse are the lower-dimensional cones of $\tgr$, which can be treated with methods similar to the ones used in Section~\ref{secResults}, and $L_{BC}$ outside the EEEE cone, for which we will use the approach presented in Section~\ref{sec:EEEE}. The results of further computations will be uploaded successively to the project's website at \url{https://software.mis.mpg.de}, together with the source code of our programs.

\section*{Acknowledgements}
The authors would like to thank Bernd Sturmfels for initiating and encouraging the collaboration for this project. They also thank the anonymous referee for valuable comments.
\vspace{-2em}

\bibliography{cubic19}

\end{document}